\documentclass[12pt,reqno]{article}
\usepackage{hyperref}
\usepackage{amsmath, amsthm, graphicx,amsfonts, amssymb,color}
\usepackage{bm}
\usepackage[inline]{enumitem}
\usepackage{verbatim}

\usepackage[notref,notcite]{showkeys}
\newtheorem{thm}{Theorem}[section]
\newtheorem{cor}[thm]{Corollary}
\newtheorem{lem}[thm]{Lemma}

\theoremstyle{definition}
\newtheorem{defn}[thm]{Definition}
\theoremstyle{remark}
\newtheorem{rem}[thm]{Remark}
\theoremstyle{example}
\newtheorem{exm}[thm]{Example}
\numberwithin{equation}{section}
\usepackage{mathrsfs}

\def\de{\end{equation}}

\def\edar{\end{eqnarray}}

\def\l{\left}\def\r{\right}
\def\lan{\langle}\def\ran{\rangle}

\def\[{\l[} \def\]{\r]}
\def\({\l(} \def\){\r)}

\def\beqlb{\begin{eqnarray}}\def\eeqlb{\end{eqnarray}}
\def\beqnn{\begin{eqnarray*}}\def\eeqnn{\end{eqnarray*}}
\def\d{{\mbox{\rm d}}}
\def\e{{\mbox{\rm e}}}

\title{\bf  Quasi-stationary distributions for time-changed symmetric $\alpha$-stable processes killed upon hitting zero}

\author{
  Zhe-Kang Fang \thanks{Laboratory of Mathematics and Complex Systems(Ministry of Education), School of Mathematical Sciences, Beijing Normal University, Beijing 100875, P.R. China.} 
  \and Yong-Hua Mao \footnotemark[1]
  \and  Tao Wang \footnotemark[1]%\thanks{School of Mathematics and Statistics, Jiangsu Normal University, Xuzhou 221000, Jiangsu,P.R. China.} 
}

\date{}

\begin{document}

%%------------------------------------------------------------
\maketitle

%%------------------------------------------------------------

\begin{abstract}
For a time-changed symmetric $\alpha$-stable process killed upon hitting zero, under the condition of entrance from infinity, we prove the existence and uniqueness of quasi-stationary distribution (QSD). The exponential convergence to the QSD from  any initial distribution is proved under conditions on transition densities.
\end{abstract}


%60G52 Stable stochastic processes
%60H20 Stochastic integral equations
%60B10 Convergence of probability measures
%60Fxx Limit theorems in probability theory

%%%%%%%%%%%%%%%%%%%%%%%%%%%%%%%%%%%%%%%%%%%%%%%%%%%%%%%%%%%%%%%%%%%%

\section{Introduction and main results}

\ \ \ \ Quasi-stationary distribution (QSD in short) is a good measurement to describe the long-time behavior of the  absorbing Markov process  when the process is conditioned to survive. Many efforts were
made to study the existence, uniqueness, the domains of attraction of  QSDs and the
convergence rate to a QSD for various Markov processes, cf. \cite{CMS13,GM15,GMZ22} for Markov chains, \cite{CCLMMS09,CMS13,LJ12} for diffusion processes, and  \cite{GNW23, HYZ19,MV12, V19} for general Markov processes under some additional conditions. 

In this paper, we will study QSDs for time-changed symmetric stable processes killed upon hitting zero.
Let $X:=(X_t)_{t\geqslant 0}$ be a symmetric $\alpha$-stable process on $\mathbb{R}$ with generator $\Delta^{\alpha / 2}:=-(-\Delta)^{\alpha / 2}$, $\alpha\in (1,2)$, where $-(-\Delta)^{\alpha / 2}$ is the fractional Laplacian. Consider the following stochastic differential equation:
\begin{equation}\label{SDE}
\mathrm{d} Y_{t}=\sigma\left(Y_{t-}\right) \mathrm{d} X_{t},
\end{equation}
where $\sigma$ is a strictly positive continuous function on $\mathbb{R}$. By \cite[Proposition 2.1]{DK20}, there is a unique weak solution $Y=(Y_t)_{t\geqslant 0}$ to the SDE \eqref{SDE}, and $Y$ can also be expressed as a  time-changed process $Y_{t}:= X_{\zeta_{t}},$ {where}  $$\zeta_{t}:=\inf \left\{s > 0: \int_{0}^{s} \sigma\left(X_{u}\right)^{-\alpha} \mathrm{~d} u>t\right\}. 
$$
By \cite[Remark 43.12]{Sa99}, the process is pointwise recurrent. Suppose $\mu(\mathrm{d} x):=\sigma(x)^{-\alpha} \mathrm{d} x$ is a probability measure.

Let $T_0=\inf\{t>0: Y_t = 0\}$, $T_\infty =\lim_{R\rightarrow +\infty} T_{(-R,R)^c}$, and $T_A=\inf\{t>0, Y_t\in A\}$, for any Borel subset $A \subseteq \mathbb{R}$. According to \cite[Theorem 2.3]{KA18} and the strong Markov property, a direct calculation (see Appendix for more details) leads to the conclusion that for any $x \neq 0$, $\mathbb{P}_x[T_0<T_\infty]=1$, which means the processes are almost-surely absorbed by $0$.  Let $Y^0$ be the (sub-)process of  $Y$ killed upon 0, with transition function
\begin{equation*}
	P_t^0 (x,A) = \mathbb{P}_x[Y_t\in A, t<T_{0}]\text{, for any }x\in \mathbb{R}^0:=\mathbb{R}\setminus\{0\}   \text{ and }A\in \mathscr{B}(\mathbb{R}^0). 
\end{equation*}

We call a probability measure $\nu$ is a QSD for $Y^{0}$ if for any $t\geq 0$ and any $A \in \mathscr{B}(\mathbb{R}^0) $,
\begin{equation*}
\mathbb{P}_\nu [Y_t\in A|t< T_0]=\frac{\mathbb{P}_\nu[Y_t\in A,t<T_0]}{\mathbb{P}_\nu[t<T_0]}=\nu(A),
\end{equation*}
where $\mathbb{P}_\nu(\cdot)=\int_{\mathbb{R}^0} \mathbb{P}_x(\cdot) \nu(\d x)$.

D\"oring and Kyprianou \cite{DK20} studied the entrance and exit from infinity for this process and Wang \cite{W21+} studied the exponential and strong ergodicity. The main purpose of this paper is to study the QSD for time-changed $\alpha$-stable processes killed upon zero. We will consider the existence and uniqueness of QSD, Yaglom limit, domain of attraction and the speed of convergence to the QSD.

In this paper, we will assume the following condition always holds:
\begin{equation}\label{I}
I^{\sigma,\alpha}:=\int_{\mathbb{R}} \sigma(x)^{-\alpha}|x|^{\alpha-1} \d x <\infty.
\end{equation}
By \cite[Theorem 1.4]{W21+}, \eqref{I} is equivalent to strong ergodicity  for $Y$; meanwhile, according to \cite[Table 2]{DK20}, \eqref{I} holds if and only if $\pm \infty$ are {\bf entrance from infinity}.

Denote by $L^2(\mathbb{R}^0,\mu)$ the space of square integrable measurable functions on $\mathbb{R}^0$ with respect to $\mu$. We first present the result on compactness for killed transition semigroup $(P_t^0)_{t>0}$ under the above condition.
\begin{thm}\label{compact}
	If $I^{\sigma,\alpha}<\infty$, then $(P_t^0)_{t> 0}$ is compact on $L^2(\mathbb{R}^0; \mu)$. 
\end{thm}

Under entrance from infinity \eqref{I}, by Theorem \ref{compact}, there exists a complete orthonormal set of eigenfunctions $\{ \psi_n \}_{n\geq 0}$ with $\|\psi_n\|_{L^2(\mathbb{R}^0;\mu)}=1$, such that $ P_t^0 \psi_n=e^{-\lambda_n t} \psi_n $ for any $n\geq 0$ and $t\geq 0$, where $\{\lambda_n\}_{n\geq 0}$ are eigenvalues of generator of $(P_t^0)_{t\geq 0}$, such that $0<\lambda_0< \lambda_1\leq \dots \rightarrow +\infty$,  where the positivity and simplicity of $\lambda_0$ will be proved in Appendix, Proposition A.2. The principal eigenfunction $\psi_0$ is called the {\bf ground state}.

By using the ground state $\psi_0$, we can state our results on QSDs.

\begin{thm}\label{do<->ent}
	If $I^{\sigma,\alpha} <\infty$, then $Y^0$ has a unique QSD:
	  \begin{equation}\label{nu}
	 		\nu(A)=\frac{\int_A \psi_0 \  \d \mu}{\int_{\mathbb{R} ^0} \psi_0 \ \d \mu },\ A\in \mathscr{B}(\mathbb{R}^0).
	 \end{equation}
	 Furthermore, $\nu$ is a Yaglom limit of $Y^0$, that is
	 for any $x\in \mathbb{R}^0$ and any subset $A\in\mathscr{B}\left(\mathbb{R}^0\right) $,
	 \begin{equation*}
	 \lim_{t \rightarrow \infty} \mathbb{P}_x[Y_t\in A|t<T_0]=\nu(A).
	 \end{equation*}
	 Moreover, there exists $0<C(x)<\infty$, such that
	 \begin{equation}\label{Yaglom-exp}
	 \Vert \mathbb{P}_x[Y_t \in \cdot | t<T_0] -\nu \Vert_{TV} \leq C(x)e^{-(\lambda_1-\lambda_0)t},
	 \end{equation}
	 where $\Vert \eta \Vert_{TV}:= \sup_{|f|\leq 1} |\eta(f)| $ denotes the total variation of a signed measure $\eta$. 
	
\end{thm}

Next, under some additional assumptions of transition density functions, we can prove  $\nu$ attracts all probability measures $\eta$ on $\mathbb{R}^0$, and the exponential convergence in total variation.

\begin{thm}\label{do_con}
	 Assume $I^{\sigma,\alpha}<\infty$. Let $$p_t^0(x,y):=\frac{\d P_t^0(x,\cdot)}{\d \mu}(y).$$
	If for any $r_0>0$,   $p_t^0(x,y)$  satisfies that $\sup_{x\in [-r_0,r_0]\setminus\{0\}} p_2^0(x,x) <+\infty$, then $\nu$ attracts all probability measures $\eta$ on $\mathbb{R}^0$, that is, for any subset $A\in\mathscr{B}(\mathbb{R}^0)  $,
	\begin{equation}\label{con_ordin}
	\lim_{t \rightarrow \infty} \mathbb{P}_\eta[Y_t\in A | t<T_0]=\nu(A).
	\end{equation}
	Furthermore, if $\sup_x p_2^0(x,x) <+\infty$, then for any probability measure $\eta$ on $\mathbb{R}^0$, 
	\begin{equation}\label{con_eta}
		\Vert \mathbb{P}_\eta [Y_t \in \cdot | t<T_0] -\nu(\cdot) \Vert_{TV}\leq \frac{C e^{-(\lambda_1-\lambda_0)t}\Vert\eta-\nu\Vert_{TV}}{\eta(\psi_0)},
	\end{equation}
	where $C$ is a constant independent of $\eta$.
\end{thm}

\begin{rem}
	 The Lyapunov function condition is also an important method for proving the existence and uniqueness of QSD, and we refer the reader to \cite[Theorem 2.2]{GNW23} for the recent result. Note that the Lyapunov function condition \cite[C3]{GNW23} implies 
	$$\lim_{n\rightarrow \infty} \sup_{x\in \mathbb{R}} \mathbb{E}_x[T_{[-n,n]}]=0,$$
	which is equivalent to entrance from infinity by Lemma \ref{equ_con} in this paper. 
\end{rem}

Note that when $\sigma$ is a polynomial, the conditions in above theorem can be explicitly characterized, and we obtain the following conclusion. 

\begin{exm}\label{eg}
	Consider the polynomial case: $\sigma(x)=\left(\frac{2}{\alpha\gamma-1}\right)^{1/\alpha}(1+|x|)^{\gamma}$ (where $\left(\frac{2}{\alpha\gamma-1}\right)^{1/\alpha}$ is the normalizing constant so that the $\mu$ is a probability measure). When $\gamma>1$, $Y^0$ has a unique QSD $\nu$ given by \eqref{nu},  $\nu$ is the Yaglom limit of $Y^0$ and \eqref{Yaglom-exp} holds. Furthermore, \ref{con_eta} holds if and only if $\gamma>1$. 
\end{exm}

An outline of this paper is as follows. In Section 2, we recall some basic notions and properties about killed processes and Green functions. In section 3,  we prove the compactness of the semigroup of killed processes (Theorem \ref{compact}), and the strict positivity, continuity and boundedness of ground state, which plays an important role in subsequent proofs. In section 4, we prove our main results on QSDs (Theorem \ref{do<->ent}, Theorem \ref{do_con} and Example \ref{eg}), and give some corollaries.

\section{{Killed processes and Green functions}}

Recall that $X$ is a one-dimensional symmetric $\alpha$-stable process with $\alpha\in(1,2)$, and $Y$ is its time-changed process.   According to \cite[Theorem 6.2.1]{FOT11}, it is known that $Y$ is $\mu$-symmetric. 
From \cite[Section 1]{CW14}, the Dirichlet form $(\mathscr{E},\mathscr{F})$ associated with $Y$ is given by
\begin{equation*}
	\mathscr{E}(f,g)=\frac{1}{2}\int_{\mathbb{R}} \int_{\mathbb{R}} (f(x)-f(y))(g(x)-g(y)) \frac{C_\alpha \d x \d y}{|x-y|^{1+\alpha}},
\end{equation*}
and
\begin{equation*}
	\mathscr{F}=\{ f\in L^2(\mu): \mathscr{E}(f,f)<\infty  \},
\end{equation*}
{where $C_\alpha=\frac{\alpha2^{\alpha-1}\Gamma((\alpha+1)/2)}{\sqrt{\pi}\Gamma(1-{\alpha}/{2})}$.}

 Given an open subset $B \subseteq \mathbb{R}$, let $Y^B$ be the killed (sub-)process of $Y$ killed upon exiting $B$ with  transition function
\begin{equation*}
P_t^B (x,A) = \mathbb{P}_x[Y_t\in A, t<T_{B^c}]\text{, for any }x\in B   \text{ and }A\in \mathscr{B}(B). 
\end{equation*}
{Define
$$ \mathscr{F}^B=\{f\in \mathscr{F}, \widetilde{f}=0,\ \text{q.e. on}\ B^c \},\quad \mathscr{E}^B=\mathscr{E} \ \text{on}\  \mathscr{F}^B\times \mathscr{F}^B,$$
where q.e. stands for quasi-everywhere, 
and $\widetilde{f}$ is a quasi-continuous modification of $f$ 
(cf. \cite[Section 2.2]{O13}). By \cite[Theorem 3.5.7]{O13},  $(\mathscr{E}^B,\mathscr{F}^B)$ is a regular Dirichlet form on $L^2(B;\mu)$ associated with $Y^B$; $Y^B$ is symmetric with respect to the measure $\mu|_{B}(\d x)$ (where $L^2(B,\mu)$
be the space of square integrable measurable functions on 
$B$ with respect to $\mu$).
We write $(\mathcal{L}^B,\mathcal{D}(\mathcal{L}^{B}))$ for the infinitesimal generator of
$(P_t^B)$ in $L^2(B;\mu)$, with 
$$
\mathcal{D}\left(\mathcal{L}^{B}\right):=\left\{u \in L^2(B;\mu): \lim _{t \rightarrow 0} \frac{P_t^{B} u-u}{t} \text { exists in } L^2(B;\mu)\right\},
$$
and
\begin{equation*}
\mathcal{L}^B f=\lim _{t \rightarrow 0} \frac{P_t^{\mathrm{B}} f-f}{t} \quad \text { in } L^2(B;\mu),\quad  \text{for any} \ f \in \mathcal{D}(\mathcal{L}^{B}).
\end{equation*}

}

 Let
\begin{equation*}
G^B(x,\d y)=\int_{0}^{+\infty} P_t^B(x,\d y) dt
\end{equation*}
 be the Green potential measure and denote the Green operator by
$$G^Bf(x)=\int_{B}f(y)G^B(x,\d y).$$ Simultaneously, the killed process $X^B$ and its Green potential measure $G_X^B(x,\d y)$ are defined similarly. Let $G_X^B(x,y)$ be the Green function of $X^B$, that is, for any $x,y \in B$, $G_X^B(x,\d y)=G_X^B(x,y) \d y$. 
According to \cite[Section 2]{W21+}, it should be pointed out here that the Green operator of $Y^B$ has a strong relationship with the Green operator of $X^B$:
\begin{equation}\label{Y_X}
G^Bf(x)= \int_B f(y) G_X^B(x,y) \sigma(y)^{-\alpha} \d y, \ \forall f \in \mathscr{B}(\mathbb{R}). 
\end{equation}

On some occasions, the Green function of $X^B$ can be expressed explicitly, for example:
\begin{enumerate}[label=(\arabic*)]
	\item {\cite[Lemma 4]{bz06}} $B=\mathbb{R}^0$: for any $x,y \neq 0$,
	\begin{equation}\label{G_X^0}
	G_X^0(x,y):=G_X^{\{0\}^c}(x,y)=\frac{\omega_\alpha}{2} \left( |y|^{\alpha-1}+|x|^{\alpha-1}-|y-x|^{\alpha-1} \right),
	\end{equation}
	where $\omega_\alpha=-\frac{1}{\cos(\pi\alpha/2)\Gamma(\alpha)}.$
 By direct calculation, we have (see \cite[Page 9]{W21+})
\begin{equation}\label{esti for G0}
	G_X^0(x,y)\leq \omega_\alpha (|x|^{\alpha-1}\wedge |y|^{\alpha-1}).
\end{equation}	

	\item {\cite[(11)]{DKW20}} $B=\mathbb{R}\setminus\left[-1,1\right]^c$: for any $x,y \in \left[-1,1\right]^c$,
	\begin{equation}\label{G_X^1c}
	G_X^{\left[-1,1\right]^c}(x,y)=c_\alpha \left( |x-y|^{\alpha-1}h\left( \frac{|xy-1|}{|x-y|} \right) -(\alpha-1) h(x)h(y)      \right),  
	\end{equation}
	where $ c_\alpha=2^{1-\alpha}/\left({\Gamma(\alpha/2)^2}\right),  $ and 
	$$h(x)=\int_{1}^{|x|}\left( z^2-1 \right)^{\frac{\alpha}{2}-1}dz.$$
\end{enumerate}
	Furthermore, {according to \cite[Lemma 3.3]{DKW20},}
	\begin{equation}\label{G_X^1clim}
	\lim_{x \rightarrow \infty} G_X^{\left[-1,1\right]^c}(x,y)=K_\alpha h(y),
	\end{equation}
	where $K_\alpha$ is a constant and defined by
	\begin{equation*}
	K_\alpha=\frac{2c_\alpha (1-\frac{\alpha}{2}) \Gamma(\frac{\alpha}{2})}{\Gamma(1-\frac{\alpha}{2})} \int_{1}^{\infty} \frac{h'(v)}{1+v} \d v <\infty.
	\end{equation*}
Besides, thanks to the self-similarity of $X$, for any $R>0$,
\begin{equation}\label{sel_1}
G_X^{\left[-R,R\right]^c}(x,y)=R^{\alpha-1}G_X^{\left[-1,1\right]^c}(\frac{x}{R},\frac{y}{R}).
\end{equation}

\section{Compactness of killed semigroups and properties of the ground states}

Let $P_t^0:=P_t^{\{0\}^c}$ be the semigroup of $Y$ killed upon hitting 0.  Denote by $Y^0:=Y^{\{0\}^c}$, $\mathcal{L}^0:=\mathcal{L}^{\{0\}^c}$ and  $G^0:=G^{\{0\}^c}$ the killed process, the generator of $P_t^0$, and the Green operator on $\mathbb{R}^0$ respectively. In this section, we first prove that under the condition $I^{\sigma,\alpha}<\infty$, $(P_t^0)_{t> 0}$ is compact; therefore, we can study the properties of ground state, which is crucial to our proofs of main results on QSDs.

Firstly, we prove Theorem \ref{compact} and $G^0$ is a Hilbert-Schmidt operator on $L^2(\mathbb{R}^0;\mu)$ under the condition $I^{\sigma,\alpha}<\infty$.

\begin{proof}[Proof of Theorem \ref{compact}]
By \cite[Theorem 0.3.9]{wfy05}, to demonstrate $(P_t^0)_{t> 0}$ is compact, we  just need to prove $0$ is belong to resolvent set $\rho(\mathcal{L}^0)$ of $\mathcal{L}^0$ and $G^0$ is compact. 
	
Firstly, we prove $0\in\rho(\mathcal{L}^0)$,  that is, the inverse of $-\mathcal{L}^0$, $G^0=(-\mathcal{L}^0)^{-1}$ is bounded in the operator norm from $L^2(\mathbb{R}^0;\mu)$ to $L^2(\mathbb{R}^0;\mu)$: 
	$$\|G^0\|_{L^2(\mathbb{R}^0;\mu)\rightarrow L^2(\mathbb{R}^0;\mu)}<\infty.$$
Note that by \eqref{esti for G0},
\begin{equation*}
	\begin{split}
		\|G^0f\|_{L^1(\mathbb{R}^0;\mu)}&=\int_{\mathbb{R}^0}\int_{\mathbb{R}^0}G^0(x,y)f(y)\mu(\d y)\mu(\d x)\\
		&\leq \int_{\mathbb{R}^0}\int_{\mathbb{R}^0}\omega_\alpha |x|^{\alpha-1}f(y)\mu(\d y)\mu(\d x)\\
		&=\omega_\alpha I^{\sigma,\alpha}\|f\|_{L^1(\mathbb{R}^0; \mu)}.
xa	\end{split}
\end{equation*}
Therefore, $\|G^0\|_{L^1(\mathbb{R}^0;\mu)\rightarrow L^1(\mathbb{R}^0;\mu)}\leq \omega_\alpha I^{\sigma,\alpha}<\infty$, and  $$\|G^0\|_{L^\infty(\mathbb{R}^0;\mu)\rightarrow L^\infty(\mathbb{R}^0;\mu)}=\|G^0\|_{L^1(\mathbb{R}^0;\mu)\rightarrow L^1(\mathbb{R}^0;\mu)}<\infty.$$
From the Riesz-Thorin theorem, it follows that $\|G^0\|_{L^2(\mathbb{R}^0;\mu)\rightarrow L^2(\mathbb{R}^0;\mu)}<\infty$. By the definition of resolvent, $0\in\rho(\mathcal{L}^0)$.

Next, we show that $G^0$ is a Hilbert-Schmidt operator, which implies that $G^0$ is compact.

Note that 
\begin{equation}
	\begin{split}
	\int_{\mathbb{R}^0}	\int_{\mathbb{R}^0}(G^0(x,y))^2\mu(\d x)\mu(\d y)&\leq 	\int_{\mathbb{R}^0}	\int_{\mathbb{R}^0} \omega_\alpha^2 (|x|^{\alpha-1}\wedge |y|^{\alpha-1})^2\mu(\d x)\mu(\d y)\\
	&\leq \omega_\alpha^2\int_{\mathbb{R}^0}|x|^{\alpha-1}\mu(\d x)\omega_\alpha\int_{\mathbb{R}^0}|y|^{\alpha-1}\mu(\d y)=(\omega_\alpha I^{\sigma,\alpha})^2<\infty,
	\end{split}
\end{equation}
therefore $G^0(x,y)\in L^2(\mathbb{R}^0\times\mathbb{R}^0;\mu\times\mu)$, and hence $G^0$ is a Hilbert-Schmidt operator.\par 
\end{proof}

Theorem \ref{compact} indicates that there exists a complete orthonormal set of eigenfunctions $\{ \psi_n \}_{n\geq 0}$ with $\|\psi_n\|_{L^2(\mathbb{R}^0;\mu)}=1$, such that $ P_t^0 \psi_n=e^{-\lambda_n t} \psi_n $ for any $n\geq 0$ and $t\geq 0$, where $\{ \lambda_n \}_{n\geq 0}$ are eigenvalues of generator of $(P_t^0)_{t\geq 0}$, satisfying $0<\lambda_0< \lambda_1\leq \dots \rightarrow +\infty$, where the positivity and simplicity of  $\lambda_0$  will be proved in Appendix, Proposition A.2. The principal eigenfunction $\psi_0$ is called the {\bf ground state}. Next, we are going to prove some basic properties about $\psi_0$. We prove the positivity and continuity of $\psi_0$ on $\mathbb{R}^0$ in the following theorem.

\begin{thm}\label{pos_con}
	If $I^{\sigma,\alpha}<\infty$, then $\psi_0$ can be chosen to be strictly positive and continuous on $\mathbb{R}^0$.
\end{thm}	
\begin{proof}
Firstly, we prove $\psi_0$ can be chosen to be nonnegative. 

Actually,  by using the proof by contradiction, we have $\mu(\{x: \psi_0(x)\geq 0 \})=1$  or $\mu(\{x: \psi_0(x) \leq 0 \})=1$. Assume that { $\mu(\{x:\psi_0(x)>0\})>0$ and $\mu(\{x:\psi_0(x)<0\})>0$. Then $\mathrm{Leb}(\{(x,y): \psi_0(x)\psi_0(y)<0\})>0$, where $\mathrm{Leb}(\cdot)$ is the Lebesgue measure.}
	Note that on $\{(x,y): \psi_0(x)\psi_0(y)\geq 0\}$,
	$$(\psi_0(x)-\psi_0(y))^2=(|\psi_0(x)|-|\psi_0(y)|)^2,$$
	while on $\{(x,y): \psi_0(x)\psi_0(y)< 0\}$,
	$$ (\psi_0(x)-\psi_0(y))^2 > (|\psi_0(x)|-|\psi_0(y)|)^2.$$
	Therefore,
	\begin{equation}\label{inf contradiction}
		\begin{split}
			\mathscr{E}(\psi_0,\psi_0)&=\frac{1}{2} \left( \int_{\{(x,y): \psi_0(x)\psi_0(y)\geq 0\}} + \int_{\{(x,y): \psi_0(x)\psi_0(y)< 0\}} \right) (\psi_0(x)-\psi_0(y))^2 \frac{C_\alpha \d x \d y}{|x-y|^{1+\alpha}}\\
			&>\frac{1}{2} \left( \int_{\{(x,y): \psi_0(x)\psi_0(y)\geq 0\}} + \int_{\{(x,y): \psi_0(x)\psi_0(y)< 0\}} \right) (|\psi_0(x)|-|\psi_0(y)|)^2 \frac{C_\alpha \d x \d y}{|x-y|^{1+\alpha}}\\
			&=\mathscr{E}(|\psi_0|,|\psi_0|).
		\end{split}
	\end{equation}
	However, by the definition of ground state and the first Dirichlet eigenvalue,
	\begin{equation*}
	\mathscr{E}(\psi_0,\psi_0)=\lambda_0=\inf\{\mathscr{E}(f,f): \mu(f^2)=1, f(0)=0, f\in \mathscr{F}\}\leq \mathscr{E}(|\psi_0|,|\psi_0|),
	\end{equation*}
	which contradicts \eqref{inf contradiction}. Thus  $\mu\{x:\psi_0(x)>0\}=0$ or $\mu\{x:\psi_0(x)<0\}=0$. If $\mu\{x:\psi_0(x)>0\}=0$, then $-\psi_0$ satisfies $P_t^0 (-\psi_0)=e^{-\lambda_0 t}(-\psi_0)$ and $\mu\{x:-\psi_0(x)< 0\}=0$.
	So we can always choose eigenfunction $\psi_0$ satisfying
	\begin{equation}\label{psi_0<0=0}
      \mu(\{x:\psi_0(x)\geq0\})=1.
	\end{equation}
	Note that by the definition of ground state, it is easy to see that
	\begin{equation*}
		G^0\psi_0(x)=\int_{0}^{\infty}P_t^0\psi_0(x)\d t=\int_{0}^{\infty}\e^{-\lambda_0 t}\psi_0(x)\d t=\frac{\psi_0(x)}{\lambda_0}.
	\end{equation*}
	Thus by combining the above equality, \eqref{Y_X}, and \eqref{psi_0<0=0}, we get that for any $x\neq 0$, 
	\begin{equation}\label{dom con thm}
		\frac{\psi_0(x)}{\lambda_0}=\int_{\mathbb{R}^0} G_X^0(x,y) \psi_0(y) \mu(\d y)\geq 0,
	\end{equation}
	which proves $\psi_0$ is nonnegative.

Next, we prove $\psi_0$ is strictly positive. 

Actually, 
\begin{equation}\label{positive Green}
G_X^0(x,y)>0\text{ for any }x, y \neq 0.
\end{equation}
According to \eqref{G_X^0}, by noting that $G_X^0(x,y)=G_X^0(-x,-y)$, we just need to prove that for any $x>0$, $y\neq 0$, $G_X^0(x,y)>0$. Indeed, $G_X^0(x,\cdot)$ is strictly decreasing in $(-\infty,0), \ (x,+\infty)$ and strictly increasing in $(0,x)$. Since for any $y>x$, by \cite[Page 592]{W21+},
$$G_X^0(x,y)>\frac{\omega_\alpha}{2} (x\wedge y)^{\alpha-1}= \frac{\omega_\alpha}{2} x^{\alpha-1}>0,$$ and $G_X^0(x,0)=0$, then  $G_X^0(x,y)>0$ for any $y\neq 0,\ x>0$. 

Now, if there would exist $x_0\neq0$ such that $\psi_0(x_0)=0$, then by \eqref{dom con thm},
\begin{equation*}
	0=\frac{\psi_0(x_0)}{\lambda_0}=\int_{\mathbb{R}^0} G_X^0(x_0,y) \psi_0(y) \mu(\d y).
\end{equation*}
By \eqref{positive Green}, and the non-negativity of $\psi_0$,  we would have $\psi_0=0$, $\mu$-a.e., which is contradictory to $ \Vert \psi_0 \Vert _{L^2(\mathbb{R}^0;\mu)}=1$. So $\psi_0$ is strictly positive on $\mathbb{R}^0$.

Finally, to prove the continuity of $\psi_0$ on $\mathbb{R}^0$, w.l.o.g., we assume $x_0>0$.
According to \eqref{esti for G0}, for any $|x-x_0|<1\wedge \frac{x_0}{2}$,
\begin{equation*}
	\int_{\mathbb{R}^0} G_X^0(x,y) \psi_0(y) \mu(\d y)\leq \omega_\alpha \int_{\mathbb{R}^0} (|x|\wedge|y|)^{\alpha-1} \psi_0(y) \mu(dy) \leq \omega_\alpha(x_0+1)^{\alpha-1} \mu(\psi_0)<\infty.
\end{equation*}
By using \eqref{dom con thm}, the continuity of $G(\cdot, y)$ on $\mathbb{R}^0$ for any $y\neq 0$ and dominated convergence theorem,
 we have thus proved Theorem \ref{pos_con}.
\end{proof}

Next, to prove the boundedness of $\psi_0$, we need some lemmas.
\begin{lem}\label{equ_con}
	$I^{\sigma,\alpha}:=\int_{\mathbb{R}} \sigma(x)^{-\alpha}|x|^{\alpha-1} \d x <\infty$ if and only if $$\lim_{R\rightarrow \infty} \sup_{x\in \mathbb{R}} \mathbb{E}_x[T_{[-R,R]}]=0.$$
	
\end{lem}
\begin{proof}
    If $I^{\sigma, \alpha}<+\infty$,  from \eqref{Y_X}, \eqref{esti for G0}, 
	\begin{align*}
		\sup_x \mathbb{E}_x[T_{[-R,R]}]&=\sup_{x\in [-R,R]^c}\int_{[-R,R]^c} G^{[-R,R]^c}(x,\d y)\\
		&=\sup_{x\in [-R,R]^c} \int_{[-R,R]^c} G_X^{[-R,R]^c}(x, y) \mu(\d y)\\
		&\leq \sup_{x\in [-R,R]^c} \int_{[-R,R]^c}  G_X^0(x,y) \mu(\d y)\\
		&\leq \omega_\alpha \int_{[-R,R]^c} |y|^{\alpha-1} \mu(\d y),
	\end{align*}
	By letting $R \rightarrow +\infty$, we obtain that $\sup_x \mathbb{E}_x[T_{[-R,R]}] \rightarrow 0 $.\par
	
	On the contrary, if $I^{\sigma,\alpha}=\infty$, then for any $R>0$,  by using \eqref{G_X^1c}, \eqref{G_X^1clim} and \eqref{sel_1},
	\begin{align*}
		\sup_x \mathbb{E}_x[T_{[-R,R]}]&=\sup_{x\in [-R,R]^c} \int_{[-R,R]^c} R^{\alpha-1}G_X^{\left[-1,1\right]^c}\left(\frac{x}{R},\frac{y}{R}\right) \mu(\d y)\\
		&\geq \int_{[-R,R]^c} \liminf_{x\rightarrow +\infty} R^{\alpha-1} G_X^{\left[-1,1\right]^c}\left(\frac{x}{R},\frac{y}{R}\right) \mu(\d y)\\
		&=\int_{[-R,R]^c} R^{\alpha-1} K_\alpha h\left(\frac{y}{R}\right) \mu(\d y)\\
		&\geq \frac{K_\alpha}{\alpha}\int_{[-R,R]^c} R^{\alpha-1} \left(\frac{|y|^{\alpha-1}}{R^{\alpha-1}}-1\right) \mu(\d y)=+\infty. 
	\end{align*}
\end{proof}
\begin{rem}\label{equ_rem}
	{According to the proof of \cite[Theorem 1.4]{W21+}, it is easy to verify that   $I^{\sigma,\alpha}<+\infty $ is  equivalent to  	$\sup_{x\in \mathbb{R}} \mathbb{E}_x[T_0] <+\infty  $; now by Proposition \ref{equ_con}, we know $I^{\sigma,\alpha}<+\infty $ is also equivalent to   $\sup_{x\in \mathbb{R}} \mathbb{E}_x[T_{[-1,1]}] <+\infty $.}
\end{rem}

\begin{lem}\label{bou_lem}
If $I^{\sigma,\alpha}<\infty$, then for any $\lambda>0$, there exists a constant $R=R(\lambda)$, such that
\begin{equation*}
	\sup_{x\in \mathbb{R}^0} \mathbb{E}_x[e^{\lambda T_{[-R,R]}}] <+\infty.
\end{equation*}
\end{lem}
\begin{proof}
 It follows from Lemma \ref{equ_con} that, for any $\epsilon >0$, there exists $R=R(\epsilon)$, such that 
\begin{equation*}
	\sup_{x} \mathbb{E}_x[T_{[-R,R]}] \leq\epsilon.
\end{equation*} 
Using Markov's inequality, we have
\begin{equation*}
	\sup_{x} \mathbb{P}_x[T_{[-R,R]}>1] \leq \epsilon.
\end{equation*} 
By Markov property, for any $n\geq 2$,
$$
   \begin{aligned}
		\mathbb{P}_x\left(T_{[-R,R]}>n \right) &=\mathbb{E}_x\left[\mathbf{1}_{\left\{T_{[-R,R]}>1\right\}} \mathbf{1}_{\left\{T_{[-R,R]} \circ \theta_1>n-1 \right\}}\right] \\
		&=\mathbb{E}_x\left[\mathbf{1}_{\left\{T_{[-R,R]}>1\right\}} \mathbb{E}_{X_1}\left[\mathbf{1}_{\left\{T_{[-R,R]}>n-1 \right\}}\right]\right] \\
		& \leq \mathbb{P}_x\left(T_{[-R,R]}>1\right) \sup _{x} \mathbb{P}_x\left(T_{[-R,R]}>n-1 \right) .
	\end{aligned}
$$
Then by induction, 
\begin{equation*}
	\sup_{x} \mathbb{P}_x[T_{[-R,R]}>n] \leq \epsilon^n.
\end{equation*}
By Fubini Theorem, it comes to the fact that for any $\lambda>0$, we can take $\epsilon <e^{-\lambda} $ and $R=R(\epsilon)$ such that for any $x\in\mathbb{R}^0$,
\begin{align*}
	\mathbb{E}_x[e^{\lambda T_{[-R,R]}}]
	&=\int_{0}^{+\infty} \lambda e^{\lambda s} \mathbb{P}_x[T_{[-R,R]}>s]dt
	\leq\sum_{i=0}^{+\infty} \lambda e^{\lambda(i+1)} \mathbb{P}_x[T_{[-R,R]}>i]+1\\
	&\leq \sum_{i=0}^{+\infty} \lambda e^{\lambda(i+1)} \epsilon^i +1 \leq \frac{\lambda e^\lambda}{1-e^\lambda\epsilon}+1<+\infty.
\end{align*}
\end{proof}

Using Lemma \ref{bou_lem} and following the proof of \cite[Lemma 5.3 and Theorem 5.4]{T18}, we obtain the following corollary.
\begin{cor}\label{bou}
If $I^{\sigma,\alpha}<\infty$, then $\psi_0$ is bounded.
\end{cor}

\section{Proof of the main results on QSD}

We verify the main results on QSD in this section. Firstly, we prove Theorem \ref{do<->ent}, which shows the existence and uniqueness of the QSD for $Y^0$, the existence of Yaglom limit and the exponential convergence to Yaglom limit when starting at a single point $x\neq 0$. Secondly, we provide a sufficient condition for exponential convergence to the QSD for any initial distribution (Theorem \ref{do_con}). Finally, we focus on Example \ref{eg} , which indicates the condition $I^{\sigma,\alpha}<+\infty$ is a sufficient and necessary condition for uniform exponential convergence on some occasions.  \par 
First of all, we prove Theorem \ref{do<->ent}. 

\begin{proof}[Proof of Theorem \ref{do<->ent}]
	(1) Firstly, we prove \begin{equation*}
		\nu(\text{d}x)=\frac{\psi_0(x)\mu(\d x)}{\int_{\mathbb{R}^0}\psi_0(x)\mu(\d x)}.
	\end{equation*} is a QSD for $Y^0$.

	Since $\psi_0 \in L^2(\mathbb{R}^0,\mu) $ and $\mu$ is a finite measure, $\psi_0 \in L^1(\mathbb{R}^0,\mu)$. According to \cite[Lemma 4.1.3]{FOT11} and the $\mu$-symmetry of $Y$, we know that $Y^0$ is $\mu$-symmetric.
	By using Theorem \ref{compact} and Theorem \ref{pos_con}, we get that for all $A\in \mathscr{B}(\mathbb{R}^0)$,
	\begin{align*}
		\mathbb{P}_{\nu}[Y_t\in A | t<T_0 ]&=\frac{\int_{\mathbb{R}^0} \mathbb{P}_x[Y_t\in A, t< T_0]\nu(\text{d}x)}{\int_{\mathbb{R}^0} \mathbb{P}_x[t< T_0]\nu(\text{d}x)}\\
		&=\frac{\int_{\mathbb{R}^0} P_t^0 \mathbf{1}_A(x) \psi_0(x) \mu(\text{d}x)}{\int_{\mathbb{R}^0} P_t^0 \mathbf{1}(x) \psi_0(x) \mu(\text{d}x)}\\
		&=\frac{\int_{\mathbb{R}^0} P_t^0 \psi_0(x)\mathbf{1}_A(x)  \mu(\text{d}x)}{\int_{\mathbb{R}^0} P_t^0  \psi_0(x)\mathbf{1}(x) \mu(\text{d}x)}\\
		&=\frac{\int_{\mathbb{R}^0} e^{-\lambda_0t}\psi_0(x)\mathbf{1}_A(x)\mu(\text{d}x)}{\int_{\mathbb{R}^0} e^{-\lambda_0t}\psi_0(x)\mathbf{1}(x)\mu(\text{d}x)}\\
		&=\nu(A).
	\end{align*}
	It comes to the conclusion that $\nu$ is a QSD for $Y^0$.\par
	(2) Secondly, we turn to prove the uniqueness of the QSD. Assume $\eta$ is also a QSD for $Y^0$, and there exists $A \in \mathscr{B}(\mathbb{R}^0)$ such that $\eta(A)>\nu(A)$.\par
	 From the proof of \cite[Proposition 3.5]{WZ21}, for any $ x \in \mathbb{R}$ and $t>0$, $P_t(x,\d y)$ has a density function $p_t(x,y)$ with respect to $\mu$. Using \cite[(3)]{CXW14}, we have for any $ x \in \mathbb{R}^0$ and $t>0$, $P_t^0(x,\d y)$ has a density function $p_t^0(x,y)$ with respect to $\mu$ satisfying
	\begin{equation*}
		P_t^0(x,\d y)=p_t^0(x,y) \mu(\d y), \qquad
		p_t^0(x,y)=p_t^0(y,x),\text{ for any }x,y\neq 0.
	\end{equation*}

	 From \cite[Theorem 2.2]{CMS13}, there exists $\lambda>0$, such that for any $A \in \mathbb{R}^0$, 
	\begin{equation*}
		\mathbb{P}_\eta[Y_t\in A, t<T_0]=e^{-\lambda t}\eta(A).
	\end{equation*}
	It follows that
	\begin{equation*}
		\int_{\mathbb{R}^0} p_t^0(x,y) \mathbf{1}_A(y) \mu(\d y) \eta(\d x)=e^{-\lambda t}\eta(A).
	\end{equation*}
	Then $\eta$ is absolutely continuous with respect to the Lebesgue measure. Let $\eta(\d x)=\eta(x) \d x$.
	
	By using Doob-$h$ transform, we define $Y^{\psi_0}$ with the transition semigroup as follows:
	\begin{equation*}
		P_t^{\psi_0} f=e^{\lambda_0 t} \frac{P_t^0(\psi_0 f)}{\psi_0}, \quad \text{for any } t>0. 
	\end{equation*}
	By definition of symmetry and conservativity, we know that $Y^{\psi_0}$ is $\psi_0^2\mu$-symmetric and conservative.
	Then we prove $Y^{\psi_0}$ is irreducible. Indeed, let $A { \subseteq \mathbb{R}^0}$ be a $P_t^{\psi_0}$-invariant measurable set satisfying $\psi_0^2\mu(A)>0$, according to \cite[Lemma 1.6.1]{FOT11}, $P_t^{\psi_0} \mathbf{1}_A(x)=0$ $\psi_0^2\mu$-a.e. on $\mathbb{R}^0\setminus A$ for any $t>0$. Therefore, $P_t^0 (\psi_0 \mathbf{1}_A)(x)=0$ $\psi_0^2\mu$-a.e. on $\mathbb{R}^0\setminus A$ for any $t>0$, which yields that $G^0 (\psi_0 \mathbf{1}_A)(x)=0$ $\psi_0^2\mu$-a.e. on $\mathbb{R}^0\setminus A$. Since $G^0(x,y)>0$ for any $x\neq 0$ and $y\neq 0$, $G^0 (\psi_0 \mathbf{1}_A)(x)>0$ for any $x\neq 0$. So $\psi_0^2 \mu(\mathbb{R}^0\setminus A)=0$, which means $Y^{\psi_0}$ is irreducible.
   Note that for any compact subsets $K,F\subseteq \mathbb{R}^0 $ and any $t>0$,
	\begin{equation}\label{etaK}
	\begin{split}
	\eta(K)&=\frac{\mathbb{P}_\eta [Y_t \in K, t<T_0]}{\mathbb{P}_\eta [t<T_0]}\leq \frac{\mathbb{P}_\eta[Y_t \in K, t<T_0]}{\mathbb{P}_\eta[Y_t \in F, t<T_0]}
	=\frac{\int_{\mathbb{R}^0}P_t^0 \mathbf{1}_K(x) \eta(\d x)}{\int_{\mathbb{R}^0}P_t^0 \mathbf{1}_F(x) \eta(\d x)}\\
	&=\frac{\int_{\mathbb{R}^0} \psi_0(x) P_t^{\psi_0}(\frac{\mathbf{1}_K}{\psi_0})(x)\eta(\d x)}{\int_{\mathbb{R}^0} \psi_0(x) P_t^{\psi_0}(\frac{\mathbf{1}_F}{\psi_0})(x)\eta(\d x)}.
	\end{split}
	\end{equation}
Note that by  the continuity and positivity of $\psi_0$, $$\left|\frac{\mathbf{1}_K}{\psi_0}\right|\leq \frac{1}{\inf_{x\in K}\psi_0(x)}<\infty.$$	
Combining this with the  irreducibility and conservativity of $Y^{\psi_0}$, \cite[Theorem 2.2]{T19} is valid, so
	\begin{equation*}\label{PtK}
	P_t^{\psi_0}\left(\frac{\mathbf{1}_K}{\psi_0}\right)  \rightarrow \frac{1}{ \mu(\mathbb{R}^0)} \int_{K} \psi_0  \d\mu, \ \text{a.e.} ,\ \text{as} \  t\rightarrow +\infty,
	\end{equation*}
	and
	\begin{equation*}\label{PtF}
	P_t^{\psi_0}\left(\frac{\mathbf{1}_F}{\psi_0}\right)  \rightarrow \frac{1}{ \mu(\mathbb{R}^0)} \int_{F} \psi_0  \d\mu, \ \text{a.e.} ,\ \text{as} \  t\rightarrow +\infty.
	\end{equation*}
	
	Therefore, by dominated convergence theorem, let $t\rightarrow\infty$ in \eqref{etaK}, we get that
	\begin{equation*}
	\eta(K)\leq \frac{\int_K \psi_0(x) \mu(\d x)}{\int_F \psi_0(x) \mu(\d x)}.
	\end{equation*}
	Then for any $x\neq 0$, by taking $K=[x,x+\Delta x]$ and $F\uparrow \mathbb{R}^0$, we arrive at
	\begin{equation*}
		\eta(x)=\lim_{\Delta x \rightarrow 0} \frac{\eta([x,x+\Delta x])}{\Delta x}\leq \lim_{\Delta x \rightarrow 0} \frac{\int_{[x,x+\Delta x]}\psi_0(x)\mu(\d x)}{\Delta x\mu(\psi_0)}=\frac{\psi_0(x)\sigma(x)^{-\alpha}}{\mu(\psi_0)}.
	\end{equation*} 
	It follows that for any $A\in \mathscr{B}(\mathbb{R}^0)$, $\eta(A)\leq \nu(A)$, which is a contradiction. Thus $\nu$ is the unique QSD.

(3)	Thirdly, we turn to the proof of the existence of Yaglom limit of $Y^0$. Using Theorem \ref{compact} and following the proof of \cite[Corollary 24]{MV12}, we have for any $A \in \mathscr{B}(\mathbb{R}^0)$ and any $t>2$,
\begin{align*}
&\Vert e^{\lambda_0t} P_t^0 \mathbf{1}_A-\lan\psi_0,\mathbf{1}_A\ran\psi_0\Vert_{L^2(\mathbb{R}^0;\mu)}^2\\ 
&=\Vert \sum_{n=1}^{+\infty } e^{(\lambda_0-\lambda_n)t} \lan\psi_n,\mathbf{1}_A\ran\psi_n\Vert_{L^2(\mathbb{R}^0;\mu)}^2=\sum_{n=1}^{+\infty } e^{2(\lambda_0-\lambda_n)t} \lan\psi_n,\mathbf{1}_A\ran^2\\ 
&\leq e^{2(\lambda_0-\lambda_1)(t-1)} e^{2\lambda_0} \sum_{n=0}^{+\infty } e^{-2\lambda_n} \lan\psi_n,\mathbf{1}_A\ran^2\\ 
&=e^{-2(\lambda_1-\lambda_0)(t-1)} e^{2\lambda_0} \Vert P_1^0 \mathbf{1}_A \Vert_{L^2(\mathbb{R}^0;\mu)}^2.
\end{align*}
By using H\"older inequality, we obtain
	\begin{align}\label{lim_t3}
	&|e^{\lambda_0(t-1)}\mathbb{P}_x[Y_t\in A, t<T_0]-\lan \psi_0,\mathbf{1}_A \ran e^{-\lambda_0} \psi_0(x)| \nonumber  \\
	&=|e^{\lambda_0 (t-1)} \lan P_{t-1}^0 \mathbf{1}_A, p_1^0(x,\cdot)\ran -\lan \psi_0,\mathbf{1}_A \ran \lan\psi_0, p_1^0(x,\cdot)\ran| \nonumber \\
	&\leq e^{-(\lambda_1-\lambda_0)(t-2)}e^{\lambda_0} (p_2^0(x,x))^{\frac{1}{2}}\rightarrow 0\text{ as }t\rightarrow +\infty,
	\end{align}
which yields that
	\begin{align}\label{lim_t1}
	\lim_{t \rightarrow \infty} e^{\lambda_0t} \mathbb{P}_x[Y_t \in A, t< T_0] 
	= \psi_0(x) \lan\psi_0,\mathbf{1}_A\ran,
	\end{align}
and
	\begin{equation}\label{lim_t2}
	\lim_{t \rightarrow \infty} e^{\lambda_0t} \mathbb{P}_x[t< T_0]= \psi_0(x) \lan\psi_0,\mathbf{1}\ran.
	\end{equation} 
 Therefore,
\begin{align*}
	\lim_{t \rightarrow \infty} \mathbb{P}_x[Y_t\in A| t<T_0]=\frac{\lim_{t \rightarrow \infty} e^{\lambda_0t} \mathbb{P}_x[Y_t \in A, t< T_0]}{\lim_{t \rightarrow \infty} e^{\lambda_0t} \mathbb{P}_x[t< T_0]}=\nu(A).
\end{align*}
	Hence $\nu$ is the Yaglom limit of $Y^0$.
	
	(4) Finally, we prove the exponential convergence to the Yaglom limit when starting at a single point $x\neq 0$.  
	Fix $x \neq 0$. According to the positivity (Theorem \ref{pos_con}), boundedness (Corollary \ref{bou}) of $\psi_0$, and \eqref{lim_t2}, 
	\begin{equation*}
	0<m_x:=\inf_{t>1} e^{\lambda_0(t-1)} \mathbb{P}_x[t<T_0] \leq \sup_{t>1} e^{\lambda_0(t-1)} \mathbb{P}_x[t<T_0] <\infty.
	\end{equation*} 
	 It follows from \eqref{lim_t3} that
	\begin{align}\label{esti for yaglom}
	&| \mathbb{P}_x[Y_t\in A| t<T_0]-\nu(A)  | \nonumber \\
	=&\left| \frac{e^{\lambda_0(t-1)}\mathbb{P}_x[Y_t\in A, t<T_0]}{e^{\lambda_0(t-1)} \mathbb{P}_x[t<T_0]} - \frac{e^{-\lambda_0} \psi_0(x) \lan\psi_0,\mathbf{1}_A\ran}{e^{-\lambda_0} \psi_0(x) \lan\psi_0,1\ran}     \right| \nonumber \\
	\leq& \frac{ \psi_0(x) \lan\psi_0,1\ran|e^{\lambda_0(t-1)}\mathbb{P}_x[Y_t\in A, t<T_0]-e^{-\lambda_0} \psi_0(x) \lan\psi_0,\mathbf{1}_A\ran  |}{m_x  \psi_0(x) \lan\psi_0,1\ran} \nonumber \\
	&+\frac{ \psi_0(x) \lan\psi_0,\mathbf{1}_A\ran|  e^{-\lambda_0} \psi_0(x) \lan\psi_0,1\ran-e^{\lambda_0(t-1)} \mathbb{P}_x[t<T_0]|}{m_x  \psi_0(x) \lan\psi_0,1\ran}  \nonumber \\
	\leq& \frac{2e^{-(\lambda_1-\lambda_0)(t-2)} e^{\lambda_0}  (p_2^0(x,x))^{\frac{1}{2}}   }{m_x }  .
	\end{align}
	Since 
	\begin{equation*}
	\Vert \mathbb{P}_x[Y_t \in \cdot | t<T_0] -\nu \Vert_{TV}=2\sup_{A\in \mathscr{B}(\mathbb{R}^0)} | \mathbb{P}_x[Y_t\in A| t<T_0]-\nu(A)  |,
	\end{equation*}
we arrive at \eqref{Yaglom-exp}	by choosing
	\begin{equation*}
	C(x)=\frac{4e^{2\lambda_1}(p_2^0(x,x))^{\frac{1}{2}}}{m_x}
	\end{equation*}
in \eqref{esti for yaglom}. Hence we finish the proof of Theorem \ref{do<->ent}.
\end{proof}

Next, we consider the problem about the domain of attraction of QSD and the speed of convergence. We prove Theorem \ref{do_con} as follows.  The idea of the proof benefits from \cite[Theorem 4.3]{ZH16} and \cite[Proof of Corollary 2.2.4]{V19}.

\begin{proof}[Proof of Theorem \ref{do_con}]
	Firstly, we prove the result about the domain of attraction.  We assume that $I^{\sigma,\alpha}<+\infty$ and for any $r_0>0$, $\sup_{x\in [-r_0,r_0]\setminus\{0\}} p_2^0(x,x) <+\infty$.\par 
	By taking $\lambda=\lambda_0$ in Lemma \ref{bou_lem}, there exists $R_0>0$ such that $$B_1:=\sup_{x\in \mathbb{R}^0} \mathbb{E}_x[e^{\lambda_0 T_{[-R_0,R_0]}}]<+\infty.$$
	Combining Corollary \ref{bou} and \eqref{lim_t3}, we have
	$$B_2:=\sup_{t>2} \sup_{x\in [-R_0,R_0]} e^{\lambda_0 t} \mathbb{P}_x[T_0>t]<+\infty. $$
	Then for any $x>R_0$, using strong Markov property, we get that
	\begin{align*}
		\mathbb{P}_x[T_0>t]&=\mathbb{P}_x[T_{[-R_0,R_0]}>t]+\mathbb{P}_x[T_{[-R_0,R_0]}\leq t, T_0>t]\\
		&\leq e^{-\lambda_0 t} \mathbb{E}_x[e^{\lambda_0 T_{[-R_0,R_0]}}] + 
		\mathbb{E}_x \left[ \mathbb{P}_x[T_{[-R_0,R_0]}\leq t,T_0>t | \mathscr{F}_{T_{[-R_0,R_0]}}]   \right]\\
		&\leq e^{-\lambda_0 t} B_1+ \mathbb{E}_x \left[ \mathbf{1}_{\{T_{[-R_0,R_0]} \leq t \}} \mathbb{P}_x[T_0>t | \mathscr{F}_{T_{[-R_0,R_0]}}]  \right]
		\\
		&\leq e^{-\lambda_0 t} B_1 + \int_{0}^{t} \int_{-R_0}^{R_0} \mathbb{P}_y[T_0>t-u] \mathbb{P}_x[T_{[-R_0,R_0]} \in \mathrm{d}u, X_{T_{[-R_0,R_0]}} \in \mathrm{d}y ]\\
		&\leq e^{-\lambda_0 t} B_1 + \int_{0}^{t} \int_{-R_0}^{R_0} B_2 e^{-\lambda_0(t-u)} \mathbb{P}_x[T_{[-R_0,R_0]} \in \mathrm{d}u, X_{T_{[-R_0,R_0]}} \in \mathrm{d}y ] \\
		&= e^{-\lambda_0 t} B_1+ e^{-\lambda_0 t} B_2 \int_{0}^{t} e^{\lambda_0 u} \mathbb{P}_x[T_{[-R_0,R_0]} \in \mathrm{d}u] \\
		&\leq e^{-\lambda_0 t} B_1+ e^{-\lambda_0 t} B_2 \mathbb{E}_x[e^{\lambda_0 T_{[-R_0,R_0]} }] \\
		&\leq e^{-\lambda_0 t} B_1(B_2+1)<+\infty.
	\end{align*} 
	Then by the above analysis,
	\begin{equation*}
		\sup_{x\in \mathbb{R}^0} \sup_{t>2}	e^{\lambda_0(t-1)} \mathbb{P}_x[Y_t \in A, t< T_0] \leq \sup_{x\in \mathbb{R}^0}	\sup_{t>2} e^{\lambda_0(t-1)} \mathbb{P}_x[t< T_0] <+\infty.
	\end{equation*}
	Therefore, by using dominated convergence theorem and \ref{lim_t1}, we have
	\begin{align*}
		\lim_{t \rightarrow \infty} e^{\lambda_0 t} \mathbb{P}_\eta[Y_t \in A, t< T_0]
		&=\lim_{t \rightarrow \infty} \int_{\mathbb{R}^0} e^{\lambda_0 t} \mathbb{P}_x[Y_t \in A, t< T_0] \eta(\d x)\\
		&=\int_{\mathbb{R}^0} \lim_{t \rightarrow \infty} e^{\lambda_0 t} \mathbb{P}_x[Y_t \in A, t< T_0] \eta(\d x)\\
		&=\int_{\mathbb{R}^0} \psi_0(x) \lan\psi_0,\mathbf{1}_A\ran \eta(\text{d}x),
	\end{align*}
where we use \eqref{lim_t1} in the last equality. 
	{Similarly,}
	\begin{equation*}
		\lim_{t \rightarrow \infty} e^{\lambda_0t} \mathbb{P}_\eta[t< T_0]=\int_{\mathbb{R}^0}  \psi_0(x) \lan\psi_0,1\ran \eta(\text{d}x).
	\end{equation*}
	{Thus we arrive at the conclusion that}
	\begin{align*}
		\lim_{t \rightarrow \infty} \mathbb{P}_\eta[Y_t \in A|t< T_0]=\frac{\lim_{t \rightarrow \infty} e^{\lambda_0 t} \mathbb{P}_\eta[Y_t \in A, t< T_0]}{\lim_{t \rightarrow \infty} e^{\lambda_0 t} \mathbb{P}_\eta[t< T_0]}=\nu(A).
	\end{align*}
	
	Secondly, we prove the result about the speed of convergence.
		
Assume that $I^{\sigma,\alpha}<+\infty$ and $\sup_x p_2^0(x,x) <+\infty$.
Let $h_t(x)=e^{\lambda_0 t}\mathbb{P}_x[T_0>t]$ and $h(x)=\psi_0(x)\mu(\psi_0)$. Then $\eta(h_t)=e^{\lambda_0 t} \mathbb{P}_\eta [T_0>t]$ and $\eta(h)=\eta(\psi_0)\mu(\psi_0)$.

We observe that
\begin{align}\label{con_4}
	\left\| \mathbb{P}_\eta[Y_t\in \cdot| t<T_0]-\nu \right\|_{TV} \nonumber 
	&=\left\| \frac{\mathbb{P}_\eta[Y_t\in \cdot, t<T_0]}{\mathbb{P}_\eta[t<T_0]}-\nu\right\|_{TV} \nonumber \\
	&=\left\| \frac{e^{\lambda_0 t}\eta P_t^0}{\eta(h_t)}-\nu \right\|_{TV} \nonumber \\
	&=\left\| e^{\lambda_0 t} \eta P_t^0 (\frac{1}{\eta(h_t)}-\frac{1}{\eta(h)}+\frac{1}{\eta(h)})-\nu\right\|_{TV} \nonumber \\
	&\leq \left\| \frac{e^{\lambda_0 t} \eta(h-h_t) \eta P_t^0}{\eta(h_t)\eta(h)} \right\|_{TV}
	+ \left\| \frac{e^{\lambda_0 t} \eta P_t^0-\eta(h) \nu}{\eta(h)}\right\|_{TV},
\end{align}
{where $\eta P_t^0(\cdot):=\int_{\mathbb{R}^0}P_t^0(x,\cdot)\eta(\d x )$.
Next, we prove the result by estimating the last two items of \eqref{con_4}}. 

According to \eqref{lim_t3}, if $\sup_x p_2^0(x,x) <+\infty$, then there exists a constant $C_1$ such that for any $A \in \mathscr{B}(\mathbb{R}^0)$,
\begin{equation}\label{con_1}
	|e^{\lambda_0 t} P_t^0 \mathbf{1}_A(x)-\psi_0(x)\int_A \psi_0(y) \mu(\d y)|\leq C_1 e^{-(\lambda_1-\lambda_0)t}.
\end{equation}
Since $\nu(\d x)=\frac{\psi_0(x)\mu(\mathrm{d} x)}{\mu(\psi_0)}$, then for any probability measure $\eta$ on $\mathbb{R}^0$,
we have 
\begin{align}\label{con_2}
	&\Vert e^{\lambda_0 t} \eta P_t^0- \eta(\psi_0)\mu(\psi_0) \nu   \Vert_{TV} \nonumber \\
	&=2 \sup_{A\in \mathscr{B}(\mathbb{R}^0)} \left| e^{\lambda_0 t} \eta P_t^0(A) -\eta(\psi_0) \int_A \psi_0(y) \mu(\d y) \right| \nonumber \\
	&=2 \sup_{A\in \mathscr{B}(\mathbb{R}^0)} \left| e^{\lambda_0 t} \int_{\mathbb{R}^0} P_t^0 \mathbf{1}_A(x) \eta(\d x) -\int_{\mathbb{R}^0} \psi_0(x) \int_A \psi_0(y)\mu(\d y) \eta(\d x)  \right| \nonumber \\
	&\leq 2 \sup_{A\in \mathscr{B}(\mathbb{R}^0)}  \int_{\mathbb{R}^0} \left| e^{\lambda_0 t} P_t^0 \mathbf{1}_A(x)-\psi_0(x) \int_A\psi_0(y)\mu(\d y) \right| \eta(\d x) \nonumber \\
	&\leq 2C_1 e^{-(\lambda_1-\lambda_0)t}.
\end{align}
Let $\eta-\nu=\tilde{\eta}$. $\tilde{\eta}$ is a signed measure and its Hahn decomposition is denoted by $\tilde{\eta}=\tilde{\eta}_+-\tilde{\eta}_-$. By calculation, it is easy to prove that $\nu(h)=\nu(h_t)=1$ and $e^{\lambda_0 t} \nu P_t^0=\nu $,  so
$$\tilde{\eta}(h-h_t)=\eta(h-h_t),\qquad e^{\lambda_0 t} \eta P_t^0-\eta(h) \nu=e^{\lambda_0 t} \tilde{\eta} P_t^0-\tilde{\eta}(h) \nu.$$
Using \eqref{con_2}, it follows that 
\begin{align*}
	&\Vert e^{\lambda_0 t} \eta P_t^0-\eta(h) \nu\Vert_{TV}=\Vert e^{\lambda_0 t} \tilde{\eta} P_t^0-\tilde{\eta}(h) \nu\Vert_{TV}\\
	&\leq \Vert e^{\lambda_0 t} \tilde{\eta}_+ P_t^0-\tilde{\eta}_+(h) \nu\Vert_{TV}+\Vert e^{\lambda_0 t} \tilde{\eta}_- P_t^0-\tilde{\eta}_-(h) \nu\Vert_{TV}\\
	&\leq 2C_1	e^{-(\lambda_1-\lambda_0)t} \Vert \tilde{\eta} \Vert_{TV}.
\end{align*}
Besides,  from \eqref{con_1}, we obtain that
\begin{align*}
	\left\| \frac{e^{\lambda_0 t} \eta(h-h_t)\eta P_t^0}{\eta(h_t)} \right\|_{TV}&=\Vert   \mathbb{P}_\eta[Y_t \in \cdot |t<T_0]\tilde{\eta}(h-h_t) \Vert_{TV}\\
	&=2 \sup_{A\in \mathscr{B}(\mathbb{R}^0)}  \mathbb{P}_\eta[Y_t \in A|t<T_0]  \left| \int_{\mathbb{R} ^0}  (h(x)-h_t(x))   \tilde{\eta}(\d x) \right|\\
	&\leq 2C_1	e^{-(\lambda_1-\lambda_0)t} \Vert \tilde{\eta} \Vert_{TV}.
\end{align*}
Combining with \eqref{con_4}, it comes to the conclusion by taking $C=4C_1/\mu(\psi_0)$.
\end{proof}

{ In the following, we point out $I^{\sigma,\alpha}<+\infty$ is necessary for the QSD attracting all probability measures on $\mathbb{R}^0$.}
\begin{thm}\label{do->ent}
	If there exists a QSD $\pi$ such that for any probability measure $\eta$ on $\mathbb{R}^0$ and any subset $ A\in\mathscr{B}(\mathbb{R}^0)$,
	\begin{equation*}
		\lim_{t \rightarrow \infty} \mathbb{P}_\eta[Y_t\in A | t<T_0]=\pi(A),
	\end{equation*}
	then $I^{\sigma,\alpha}<\infty$. 
\end{thm}

\begin{proof}
	According to \cite[Theorem 2.2]{CMS13}, there exists a constant $\beta>0$, such that 
	\begin{equation}\label{exp distri}
		\mathbb{P}_\pi[T_0>t]=e^{-\beta t}, \ \forall \ t\geq 0. 
	\end{equation}
	 By using \eqref{exp distri}, and a similar argument to the proof of  \cite[Proposition 7.5]{CCLMMS09}, we show that for any $\lambda \in (0,\beta)$ and any probability measure $\eta$ on $\mathbb{R}^0$, $\mathbb{E}_\eta[e^{\lambda T_0}]<+\infty$.
For any $x\neq 0$, by taking $\eta=\delta_x$ (where $\delta_x$ is the Dirac measure), we obtain that  $\mathbb{E}_x[e^{\lambda T_0}]<+\infty$. Let $g(x)=\mathbb{E}_x[e^{\lambda T_0}]<+\infty$. 

We claim that $g$ is bounded. If not, there would exist sequences $\{x_n\}$ such that $g(x_n)\geq 2^n, \ \forall \ n \geq 1$. However, if we take $\eta=\sum_{n=1}^{+\infty} \frac{1}{2^n} \delta_{x_n}$, 
then it is easy to verify that $\mathbb{E}_\eta[e^{\lambda T_0}]=+\infty$,  which is a contradiction. \par
{ Therefore, we arrive at
	\begin{equation*}
		\lambda \sup_{x\in \mathbb{R}^0} \mathbb{E}_x[T_0]+1 \leq \sup_{x\in \mathbb{R}^0} \mathbb{E}_x[e^{\lambda T_0}] <+\infty.
	\end{equation*}
	Combining it with Remark \ref{equ_rem}, 
	 $I^{\sigma,\alpha}<\infty$.}\end{proof}

Next, we will obtain a corollary about the exponential moments of the hitting time $T_0$, and prove the first Dirichlet eigenvalue $\lambda_{0}$ equals to the  uniform decay rate
	\begin{equation*}
		\lambda_0':=\lim_{t\rightarrow +\infty}-\frac{1}{t} \log \sup_{x\in \mathbb{R}^0} \mathbb{P}_x[T_0>t].
\end{equation*}.
\begin{cor}
	If $I^{\sigma,\alpha}<\infty$, and for any $r_0>0$, $\sup_{x\in [-r_0,r_0]\setminus\{0\}} p_2^0(x,x) <+\infty$, then
	
(1) we have	\begin{equation}\label{exp_lam}
		\sup_{x\in \mathbb{R}^0} \mathbb{E}_x[e^{\lambda T_0}] <+\infty \ \text{if and only if}\ \lambda<\lambda_0,
	\end{equation}

(2) $\lambda_{0}=\lambda_{0}'.$ 
\end{cor}
\begin{proof}
   (1) { The ``only if'' implication follows from \cite[Claim 2.4]{LJ12} and the monotonicity of exponential function.} Hence we only prove ``if'' part. By the assumptions, Theorem \ref{do_con} holds. Note that 
   	$$\mathbb{P}_\nu[T_0>t]=\int_{\mathbb{R}^0}P_t^01\d\nu=\frac{\lan P_t^01,\psi_0\ran}{\mu(\psi_0)}=\e^{-\lambda_{0} t},$$
   	 so by proof of Theorem \ref{do->ent}, we know for any $\lambda<\lambda_0$, $\sup_x\mathbb{E}_x[e^{\lambda T_0}]<+\infty$.

   (2) Using \eqref{lim_t2}, for any $x\neq 0$,
	\begin{equation*}
		\lambda_0=\lim_{t \rightarrow \infty} -\frac{1}{t} \log\mathbb{P}_x[T_0>t].
	\end{equation*}
	It follows that 
	\begin{equation*}
		\limsup_{t\rightarrow +\infty} -\frac{1}{t} \log \sup_{x\in \mathbb{R}^0} \mathbb{P}_x[T_0>t]\leq \lim_{t \rightarrow \infty} -\frac{1}{t} \log\mathbb{P}_x[T_0>t]=\lambda_0.
	\end{equation*}
For any $\lambda<\lambda_0$, using \eqref{exp_lam} and Chebyshev inequality,
	\begin{equation*}
		\sup_{x\in \mathbb{R}^0} \mathbb{P}_x[T_0>t] \leq e^{-\lambda t} \sup_{x\in \mathbb{R}^0} \mathbb{E}_x[e^{\lambda T_0}],
	\end{equation*}
which yields that
	\begin{equation*}
		\liminf_{t\rightarrow +\infty} -\frac{1}{t} \log \sup_{x\in \mathbb{R}^0} \mathbb{P}_x[T_0>t] \geq \lambda,
	\end{equation*}
	which completes the proof.
\end{proof}
It should be mentioned here that for some classical cases, such as Example \ref{eg}, the exponential convergence to QSD \eqref{con_eta} is equivalent to the {entrance from infinity}. Now we prove Example \ref{eg} as follows.

\begin{proof}[Proof of Example \ref{eg}]
	According to \cite[Corollary 6]{W21+}, $I^{\sigma, \alpha}<+\infty$ holds if and only if $\gamma>1$. Therefore, when $\gamma>1$, by Theorem \ref{do<->ent} , $Y^0$ has unique QSD $\nu$ given by \eqref{nu},  $\nu$ is the Yaglom limit of $Y^0$ and \eqref{Yaglom-exp} holds. 
	
	Furthermore, by \cite[(3.11)]{WZ21}, when $\gamma>1$, for the transition density $p_t(x,y)$ with respect to $\mu$, by choosing $t=1$, we have there exists a constant $C_2>0$ such that
	$$p_1(x,y)\leq C_2,\quad x,y\in\mathbb{R}.$$
	Note that $$p_2(x,x)=\int_{\mathbb{R}}p_1(x,y)^2\mu(\d y)\leq C_2^2, $$	thus we know that
	\begin{equation*}
	\sup_x p_2^0(x,x) \leq \sup_x p_2(x,x) <+\infty.
	\end{equation*}	
	Therefore, \eqref{con_eta} holds in this case. Noting that the exponential convergence \eqref{con_eta} implies that \eqref{con_ordin}, so by Theorem \ref{do->ent}, \eqref{con_eta} means that $\gamma>1$. Thus the exponential convergence to QSD \eqref{con_eta} is equivalent to the {entrance from infinity} in this example. 
\end{proof}

Finally, let's discuss the quasi-ergodic distribution for $Y^0$, which is a related topic.
\begin{defn}\label{def_QED}
	Let $\rho$ be a probability measure on $\mathbb{R}^0$. We say $\rho$ is a quasi-ergodic distribution(QED) for $Y^0$, if for any $x\neq 0$ and any $A\in \mathscr{B}(\mathbb{R}^0)$,
	\begin{equation*}
		\lim_{t \rightarrow \infty} \mathbb{E}_{ x} \left(  \frac{1}{t} \int_{0}^{t} \mathbf{1}_A(X_s) \d s| T_0>t    \right) =\rho(A). 
	\end{equation*}  
\end{defn}

	Using  \eqref{lim_t1}, \eqref{lim_t2} and Corollary \ref{bou}, it is easy to prove the following statement in the same way as \cite[Theorem 3.1]{HYZ19}.\par 
\begin{cor}
	Assume $I^{\sigma,\alpha}<\infty$, then for any bounded and measurable functions $f, g$ on  $\mathbb{R}^0$, $x \neq 0$  and $ 0<p<q<1$,
	\begin{enumerate}[label=(\arabic*)]
		\item $		\lim_{t \rightarrow \infty} \mathbb{E}_x[f(X_{pt}) g(X_t) | T_0>t ] =\int_{\mathbb{R}} f(y) m(\d y) \int_{\mathbb{R}} g(y) \nu(\d y)$,
		\item $		\lim_{t \rightarrow \infty} \mathbb{E}_x[f(X_{pt}) g(X_{qt}) | T_0>t ] =\int_{\mathbb{R}} f(y) m(\d y) \int_{\mathbb{R}} g(y) m(\d y)$, 
		\item $		\lim_{t \rightarrow \infty} \mathbb{E}_{ x} \left(  \frac{1}{t} \int_{0}^{t} f(X_s) \d s| T_0>t    \right) = \int_{\mathbb{R}} f(y) m(\d y)$, 
	\end{enumerate}
	where \begin{equation*}
		m(\d y)=\psi_0(y)^2 \mu(\d y) \text{ and } \nu \text{ is the QSD for } Y^0.
	\end{equation*}
\end{cor}

\section*{Appendix}
\noindent {\bf Proposition A.1}
	If $\alpha\in (1,2)$, then for all $x\neq 0$,  $\mathbb{P}_x[T_0<T_\infty]=1$.

\begin{proof}
	According to \cite[Chapter V]{BG68} and Blumenthal-Getoor-McKean theorem, it is easy to prove the following formula:
	\begin{equation}\label{non_exp}
		\mathbb{P}_x[T_0<T_{[R,+\infty)} \wedge T_{(-\infty, -R]}]=\frac{G_X^{(-R,R)}(x,0)}{G_X^{(-R,R)}(0,0)}.
	\end{equation}
	Using \cite[Theorem 2.2.3]{KA18}, it follows that
	\begin{equation*}
	G_X^{(-R,R)}(x,0)=\frac{2^{1-\alpha}|x|^{\alpha-1}}{\Gamma(\frac{\alpha}{2})^2}\int_{1}^{\frac{R}{|x|}} (s+1)^{\alpha/2-1}(s-1)^{\alpha/2-1} ds,
	\end{equation*}
	\begin{equation*}
		G_X^{(-R,R)}(0,0)=\frac{2^{1-\alpha}R^{\alpha-1}}{\Gamma(\frac{\alpha}{2})^2(\alpha-1)}.
	\end{equation*}
	It comes to the conclusion by letting $R\rightarrow +\infty$ and using L'H\^ospital's rule in \eqref{non_exp}.
\end{proof}

{ \noindent {\bf Proposition A.2} If $I^{\sigma,\alpha}<\infty$, then $\lambda_0>0$. Furthermore, $\lambda_0$ is simple.}
\begin{proof}
If $I^{\sigma,\alpha}<\infty$, then according to \cite[Theorem 1.4]{W21+}, the process $Y$ is strongly ergodic, thus $Y$ is exponentially ergodic. Therefore, by \cite[Theorem 1.1]{W21+}, $$\delta:=\sup_{x} |x|^{\alpha-1}\int_{\mathbb{R}\setminus (-|x|,|x|)}\sigma(y)^{-\alpha}\d y<\infty.$$ So by \cite[Theorem 1.3]{W21+},
$\lambda_{0}\geqslant
(4\omega_{\alpha}\delta)^{-1}>0.$

According to \cite[Theorem 6.6]{SH74} and spectral representation theorem, the spectral radius of $G^0$ is a simple eigenvalue and equals to $\lambda_0^{-1}$,  which implies $\lambda_0$ is a simple eigenvalue.	

\end{proof}

%----------------------------------------------------------------------------------------------
%----------------------------------------------------------------------------------------------

{\bf Acknowledgements}\
This work was supported by the National Nature Science Foundation of China (Grant No. 12171038), National Key Research and Development Program of China (2020YFA0712901).

%%%%%%%%%%%%%%%%%%%%%%%%%%%%%%%%%%%%%%%%%%%%%%%%%%%%%%%%%%%%%%

\bibliographystyle{plain}
\bibliography{stable}

\end{document}